\documentclass[12pt]{article}
\usepackage{graphicx}
\usepackage{amsmath}
\usepackage{amsfonts}
\usepackage{amsthm}
\usepackage{url}

\newcommand{\Real}{\mathop{\rm Re}\nolimits}
\newcommand{\Imag}{\mathop{\rm Im}\nolimits}

\begin{document}
\title{On the relation between extremal elasticity tensors with
orthotropic symmetry and extremal polynomials. }
\author{Davit Harutyunyan and Graeme Walter Milton\\
\textit{Department of Mathematics, The University of Utah}}

\maketitle
\begin{abstract}
We prove that an elasticity tensor with orthotropic symmetry is extremal if the determinant of its acoustic tensor
is an extremal polynomial that is not a perfect square.
\end{abstract}


\newtheorem{Theorem}{Theorem}[section]
\newtheorem{Lemma}[Theorem]{Lemma}
\newtheorem{Corollary}[Theorem]{Corollary}
\newtheorem{Remark}[Theorem]{Remark}
\newtheorem{Definition}[Theorem]{Definition}

\section{Introduction}
\label{sec:intro}

A necessary condition for a body containing a linearly elastic homogeneous material
with elasticity tensor $C$ to be stable when the displacement is fixed at the boundary
is the Legendre-Hadamard condition that the quadratic form associated with $C$, $f(\xi)=(C\xi;\xi)$
be rank-one convex, i.e.
$$ f(x\otimes y)=\sum_{ijk\ell=1}^3x_iy_jC_{ijk\ell}x_ky_\ell\geq 0\quad \forall x,y. $$
If one has equality for some non-zero $x,y$ then shear bands can form.
Associated with $f$ is the $y$-matrix (acoustic tensor), $T(y)$ with matrix elements
$$ T_{ik}(y)=\sum_{j\ell=1}^3y_jC_{ijk\ell}y_\ell\geq 0\quad \forall x,y, $$
so rank-one convexity is equivalent to $ T(y)$ being positive semi-definite for all $y$,
which ensures real-valued wave speeds.
The elasticity tensor $C$ need not be positive semi-definite for this condition to be satisfied, i.e $f(\xi)$ need not be convex.

In this paper we show there is an interesting connection between extremal
polynomials and extremal elasticity tensors which are at the boundary of being rank-one convex.

In general for quadratic functions $f(\xi)=(M\xi;\xi)$, with $M$ not
necessarily having the symmetry of elasticity tensors, Van Hove [\ref{bib:VanHove1},\ref{bib:VanHove2}]
proved that rank-one convexity is equivalent to the condition of
quasiconvexity introduced by  Morrey [\ref{bib:Mor1},\ref{bib:Mor2}]. Due to this, and since we are only dealing with quadratic functions
we will use the terms quasiconvexity and rank-one convexity interchangeably. (If the fields were not gradients but had different differential constraints
then rank-one convexity is no longer appropriate but quasiconvexity is appropriate, and hence we have a preference for the term  quasiconvexity). Ball
introduced the condition of polyconvexity and proved it
to be an intermediate condition between convexity and quasiconvexity [\ref{bib:Ball1}]. In the quadratic case polyconvexity is equivalent
[\ref{bib:Dac}, page 192, Lemma 5.27]
to $f$ being the sum of a convex function and a null-Lagrangian, which in the quadratic case is a function $f$ such that $f(x\otimes y)$ vanishes for all $x$ and $y$. There exist quadratic forms that are quasiconvex but not polyconvex, as shown by Terpstra in [\ref{bib:Terp2}]. Explicit examples were
given by Serre [\ref{bib:Ser1},\ref{bib:Ser2}], see also Ball [\ref{bib:Ball2}], and an especially simple example is given in [\ref{bib:Har.Mil}].
A special case of quasiconvex quadratic forms are the so called extremal ones introduced by Milton in [\ref{bib:Mil3}, page 87], see also [\ref{bib:Mil1}, section 25.2]. This and two alternative definitions of extremals were used in [\ref{bib:Har.Mil}]. In this work we will use a definition which is equivalent
to the original definition:
\begin{Definition}
\label{def:1}
A quadratic quasiconvex form is called an extremal if one cannot subtract a
rank-one form from it while preserving the quasiconvexity of the form.
\end{Definition}
If a quadratic form $f(\xi)=(C\xi;\xi)$ is extremal and does not depend on the antisymmetric part of $\xi$ we call $C$ an extremal elasticity tensor.
We prove that an elasticity tensor with orthotropic symmetry is extremal if the determinant of the $y-$matrix (acoustic tensor) is an extremal polynomial that is not a perfect square. The problem of characterizing all such extremals is a task for the future.

\section{Motivations for studying extremals}
\label{motivation}
One motivation for studying extremals comes when bounding, using the translation method, the elastic energy in a multiphase phase periodic composite
with known volume fractions of the phases. Then it is always best to use translations $C$ such that the quadratic form $f(\xi)=(C\xi;\xi)$ is an extremal
[\ref{bib:Mil3}, page 87], see also [\ref{bib:Mil1}, section 25.2].
Extremals (with an alternative definition of extremal: one cannot subtract a symmetrized rank-one form from it while preserving the quasiconvexity)
were used by Allaire and Kohn [\ref{bib:All.Kohn}] in this way to bound the elastic energy of two phase composites with isotropic phases.

Extremals may also be important for obtaining sharp geometry independent estimates with Dirichlet boundary conditions of the elastic energy stored within say a two-phase body $\Omega$ (where by geometry independent we mean independent of the distribution of the phases in the body, not independent of the shape of the body).
The ensuing analysis is an extension of the ideas of Tartar and Murat [\ref{bib:Tar0},\ref{bib:Mur.Tar},\ref{bib:Tar2}]
and Lurie and Cherkaev [\ref{bib:LurCherk1}, \ref{bib:LurCherk2}] for bounding the effective moduli of composite materials using the translation method
and that of Kang and Milton [\ref{bib:Kan.Mil}] for bounding the volume fractions of materials in a two-phase body.

Let $\tilde{C}(x)$ denote the elasticity tensor taking the positive definite value $C_1$ in phase 1 and the positive definite value $C_2$ in phase 2. With
Dirichlet boundary conditions
$\tilde{u}=u_0$ on $\partial\Omega$, the elastic energy is
\begin{equation} \tilde{W}(u_0)=\frac{1}{2}\int_{\Omega}(\tilde{C}\nabla \tilde{u};\nabla \tilde{u})\,dx
\end{equation}
where the stress $\tilde{C}\nabla \tilde{u}$ is symmetric and only depends on the strain $\tilde{\epsilon}(x)=[\nabla \tilde{u}(x)+(\nabla \tilde{u(x)})^T]/2$,
since $\tilde{C}A=0$ when $A$ is antisymmetric.
Let the quadratic form associated with $C$ be quasiconvex and chosen so that $\tilde{C}(x)-C$ is positive semidefinite,
while $C_1-C$ and/or $C_2-C$ is degenerate (on the space of symmetric matrices when $C$ has the symmetries of elasticity tensors). Writing $\tilde{C}(x)=[\tilde{C}(x)-C]+C$, we have
\begin{equation} \tilde{W}(u_0)\geq \frac{1}{2}\int_{\Omega}f(\nabla\tilde{u})\,dx
\end{equation}
with equality when $\nabla \tilde{u}$ is in the null space of $\tilde{C}(x)-C$. Suppose we are able to find one solution of the elasticity equations in the medium with tensor $C$, i.e.
\begin{equation}
\label{solution}
 \nabla\cdot\sigma=0,\quad \sigma=C\nabla u
\end{equation}
with $u=u_0$ on $\partial\Omega$: if necessary we could start with a solution to the equations (\ref{solution}) and choose $u_0$ as the boundary value of $u$. Then because
the quadratic form $C$ is quasiconvex
\begin{equation}
\label{condition}
\int_{\Omega}f(\nabla \tilde{u})\,dx\geq \int_{\Omega}f(\nabla{u})\,dx=\int_{\partial\Omega}u\cdot(\sigma\cdot n)\,dS\equiv W(u_0)
\end{equation}
where $n$ is the outward normal to the surface $\partial\Omega$. To see the condition for equality in (\ref{condition})
define $\delta u=\tilde{u}-u$ inside $\Omega$, then find a cube $B$ containing
$\Omega$, set $\delta u$ to be zero inside the remainder of the cube outside $\Omega$, and finally extend $\delta u$ to be periodic with this cube $B$ as a unit cell. Then $\nabla\delta u$ has zero average value over the unit cell and
since $u$ solves (\ref{solution}),
\begin{equation}
\int_{\Omega}f(\nabla \tilde{u})- f(\nabla{u})\,dx=\int_Bf(\nabla \delta{u})\,dx
 \end{equation}
The condition for this to be zero is easily found using the rank-one convexity and Plancherel's theorem:
each Fourier component $\widehat{\delta{u}}(k)$ of $\delta{u}$ must be such that
\begin{equation}
f(\Real\widehat{\delta{u}}(k)\otimes k)=0\quad{\rm and}\quad f(\Imag\widehat{\delta{u}}(k)\otimes k)=0
\end{equation}
where $\Real\widehat{\delta{u}}(k)$ and $\Imag\widehat{\delta{u}}(k)$ are the real and imaginary parts of $\widehat{\delta{u}}(k)$.
Fields $\nabla\delta u$ satisfying this condition are called special fields and a necessary condition for them to exist is that $C$ not be strictly
quasiconvex.

In summary we have the inequality
\begin{equation}
\tilde{W}(u_0)\geq W(u_0),
\end{equation}
which will be sharp when $\nabla \tilde{u}$ is in the null space of $\tilde{C}(x)-C$ and $\nabla\delta u$ is a special field which
vanishes in $B\setminus\Omega$. Our chances of finding such fields are greatest when $C_1-C$ and/or $C_2-C$ is especially degenerate
and when there are lots of special fields which vanish in $B\setminus\Omega$. The last condition is most likely to hold when
$f(\xi)$ is extremal, although an example has yet to be produced of an extremal function of gradients, other than a null-Lagrangian, for which there exist
special fields which vanish in $B\setminus\Omega$ when $\Omega$ is strictly contained in $B$. However, for bounding the energy stored
in a unit cell $\Omega$ of a periodic composite with periodic boundary conditions on $\nabla \tilde{u}$ (which is relevant to
bounding the effective moduli using the comparison bound) we can take $B=\Omega$ and so any special field automatically vanishes in $B\setminus\Omega$
since $B\setminus\Omega$ is empty.

\section{Orthotropic materials}
\label{sec:orthotropic}

We now briefly introduce orthotropic materials. A homogeneous orthotropic elastic material has three mutually orthogonal planes such that the
material properties are symmetric under reflection about each plane. If cartesian coordinate axes are chosen orthogonal to these planes, then the properties
are invariant under the transformation $x_a \to -x_a$, $x_b \to x_b$, and $x_c\to x_c$, where $abc$ is permutation of $123$. Elements of the elasticity tensor
such as $C_{abcc}$ and $C_{abbb}$ in general change sign under such a transformation, so these must be zero. Thus the elements $C_{ijk\ell}$ of the elasticity tensor must be zero
unless the indices $ijk\ell$ contain an even number of repetitions of the indices $1$, $2$ or $3$. Using the Voigt notation for the elements of $C$ the
constitutive law takes the form $\sigma=C\epsilon$
where
\begin{equation}
\label{stiffness.matrix}
\sigma=\begin{bmatrix} \sigma_{11} \\ \sigma_{22} \\ \sigma_{33} \\ \sigma_{23} \\ \sigma_{31} \\ \sigma_{12} \end{bmatrix}, \quad
\epsilon=\begin{bmatrix} \epsilon_{11} \\ \epsilon_{22} \\ \epsilon_{33} \\ 2\epsilon_{23} \\ 2\epsilon_{31} \\ 2\epsilon_{12} \end{bmatrix}, \quad
C=
\begin{bmatrix}
C_{11} & C_{12} & C_{13} & 0 & 0 & 0\\
C_{12} & C_{22} & C_{23} & 0 & 0 & 0\\
C_{13} & C_{23} & C_{33} & 0 & 0 & 0\\
0 & 0 & 0 & C_{44} & 0 & 0\\
0 & 0 & 0 & 0 & C_{55} & 0\\
0 & 0 & 0 & 0 & 0 & C_{66}
\end{bmatrix}.
\end{equation}
The mechanical properties are, in general, different along each axis. Orthotropic materials require 9 elastic constants and have as subclasses isotropic materials
(with 2 elastic constants), cubic materials (with 3 elastic constants), and transversely isotropic materials (with 5 elastic constants).
The wood in a tree trunk is an example of a material which is locally orthotropic: the material properties in three perpendicular
directions, axial, radial, and circumferential, are different.
Many crystals and rolled metals are also examples of orthotropic materials.
\section{Extremal polynomials and relations to the determinants of extremal quasiconvex quadratic forms}
\label{sec:extremal.polynomials}

In this section we define the notions of \textit{extremality and equivalence} of homogeneous polynomials.
\begin{Definition}
\label{def:extr.poly}
Assume $m$ and $n$ are natural numbers and $P(x_1,x_2,\dots,x_n)$ is a polynomial of degree $2m$ that is homogeneous of the same degree. Then $P(x)$ is called an extremal polynomial, if $P(x)\geq 0$ for all $x\in\mathbb R^n$ and $P(x)$ cannot be written as a sum of two other non-negative polynomials that are linearly independent.
\end{Definition}

\begin{Definition}
\label{def:pol.equiv}
Assume $m$ and $n$ are natural numbers and $P(x_1,x_2,\dots,x_n)$ and $Q(x_1,x_2,\dots,x_n)$ are polynomials of degree $2m$ that are homogeneous of the same degree. Then $P(x)$ and $Q(x)$ are equivalent if there exists a non-singular matrix $A\in\mathbb R^{n\times n}$ such that $P(x)=Q(Ax).$
\end{Definition}

It is then straightforward to prove that this notion of equivalence is actually an equivalence relation preserving also the extremality of polynomials.

\begin{Theorem}
\label{th:equiv}
The notion of equivalence introduced in Definition~\ref{def:pol.equiv} has the following properties:
\begin{itemize}
\item $P$ is equivalent to itself
\item If $P$ is equivalent to $Q$ then $Q$ is equivalent to $P$
\item If $P$ is equivalent to $Q$ and $Q$ is equivalent to $R,$ then $P$ is equivalent to $R$
\item If $P$ is equivalent to $Q$ and $Q$ is an extremal then $P$ is an extremal too
\end{itemize}

\end{Theorem}
As pointed out in introduction our future goal is describing all extremal quasiconvex quadratic forms and the first step to the
final goal has been made in [\ref{bib:Har.Mil}], where a class of extremals has been found. In order to make progress towards the goal,
one asks natural question: What are the properties of extremal quadratic forms? Such a question has not been addressed in [\ref{bib:Har.Mil}],
but in the present work for the first time. The sought property we believe is the following: The determinant of the $y$-matrix of the form must be an extremal
polynomial, which is not a perfect square. The sufficiency of that statement is proven in the present work for quadratic forms with a linear
elastic orthotropic symmetry. Let us now motivate our choice by some examples.\\

\textbf{Example 1.} The form $f(\xi)=\xi_{11}^2+\xi_{22}^2+\xi_{33}^2$ has a determinant of its $y$-matrix equal to
$y_1^2y_2^2y_3^2$ which is evidently an extremal polynomial, but $f(\xi)$ is obviously not an extremal.\\

\textbf{Example 2.} The form $f(\xi)=\xi_{11}^2+\xi_{22}^2$ has a determinant of its $y$-matrix equal to
$0,$ which is evidently an extremal polynomial, but $f(\xi)$ is obviously not an extremal.\\

\textbf{Example 3.} The determinant of the $y$-matrix of any rank-one form is equivalently zero,
but a rank-one form is not an extremal.\\

\textbf{Example 4.} The most interesting and motivating example, that contains the sought information is the extremal quasiconvex quadratic form
 $$Q(\xi)=\xi_{11}^2+\xi_{22}^2+\xi_{33}^2-2(\xi_{11}\xi_{22}+\xi_{11}\xi_{33}+\xi_{22}\xi_{33})+\xi_{12}^2+\xi_{23}^2+\xi_{31}^2,$$
that appears in [\ref{bib:Har.Mil}] (but which does not derive from a tensor having the symmetries of an elasticity tensor).

It turns out that the polynomial
$$P(y)=y_1^4y_2^2+y_2^4y_3^2+y_3^4y_1^2-3y_1^2y_2^2y_3^2$$
that is the determinant of the $y$-matrix of $Q(\xi)$ is a non-trivial extremal polynomial (by which we mean a polynomial
which is not a perfect square). Let us give a proof of that statement.
\begin{proof}
Assume in contradiction that the polynomial $P(y)$ is not an extremal. Hence there exists a polynomial $P_1(y)$ such that
\begin{equation}
\label{eq:0<P1<P}
0\leq P_1(y)\leq P(y)\quad \text{for all}\quad y\in\mathbb R^3,
\end{equation}
and $P_1(y)$ is not a multiple of $P(y).$ We aim to prove that (\ref{eq:0<P1<P}) implies $P_1=\alpha P$ for some $\alpha\in\mathbb R.$
It is clear that none of the variables $y_i$ appears in $P_1$ with power $5$ or $6$ as otherwise inequality (\ref{eq:0<P1<P})
would be violated. The coefficient of $y_1^4$ in $P_1$ is a quadratic polynomial in $y_2$ and $y_3$ that is less or equal to $y_2^2$ as inequality
(\ref{eq:0<P1<P}) implies when $y_1\to\infty,$ thus it depends only on $y_2,$ i.e., the coefficient of $y_1^4$ in $P_1$ has the form $ay_2^2.$ Similarly
the coefficients of $y_2^4$ and $y_3^4$ in $P_1$ are $by_3^2$ and $cy_1^2$ respectively. Thus $P_1$ has the form
\begin{align*}
P_1(y)&=(ay_1^4y_2^2+by_2^4y_3^2+cy_3^4y_1^2+dy_1^2y_2^2y_3^2)+a_1y_1^3y_2^3+a_2y_2^3y_3^3+a_3y_3^3y_1^3\\
&+a_4y_1^3y_2^2y_3+a_5y_1^3y_2y_3^2+a_6y_2^3y_1^2y_3+a_7y_2^3y_1y_3^2+a_8y_3^3y_1^2y_2+a_9y_3^3y_2^2y_1.
\end{align*}
We call the expression in the brackets in $P_1$ the principal part of $P_1.$
Note, that changing the sign of any of the variables $y_i$ does not change $P(y)$ but changes the sign of all summands in $P_1$ that
have an odd power of $y_i$, thus summing up the inequalities $0\leq P_1(y_1,y_2,y_3)\leq P(y_1,y_2,y_3)$ and $0\leq P_1(-y_1,y_2,y_3)\leq P(-y_1,y_2,y_3)$
we get $0\leq P_2(y)\leq P(y)$ where $P_2$ has no summands with an odd power of $y_1$ and has the same principal part as $P_1.$ Applying the same
idea to $P_2$ for the variable $y_2$ we end up with the inequality
\begin{equation}
\label{eq:0<P.principal<P}
0\leq ay_1^4y_2^2+by_2^4y_3^2+cy_3^4y_1^2+dy_1^2y_2^2y_3^2\leq P(y).
\end{equation}
It is then clear that $0\leq a,b,c\leq 1.$ We have by the Cauchy-Schwartz inequality
$$ay_1^4y_2^2+by_2^4y_3^2+cy_3^4y_1^2\geq 3(abc)^{1/3}y_1^2y_2^2y_3^2$$
and the equality holds for some choice of $y,$ thus we get $3(abc)^{1/3}\geq -d.$
On the other hand as $P(1,1,1)=0,$ then inequality (\ref{eq:0<P.principal<P}) implies $a+b+c+d=0$, thus we get
$$3(abc)^{1/3}\geq a+b+c$$ which means that the equality holds in Cauchy-Schwartz, thus $a=b=c$ and $d=-3a,$ thus the principal part
of $P_1(y)$ is a multiple of $P(y).$ On the other hand testing (\ref{eq:0<P1<P}) with $y=(1,t,0)$ we get
$$at^2+a_1t^3\geq 0\quad\text{for all}\quad t\in\mathbb R,$$
thus $a_1=0.$ Similarly we obtain $a_2=a_3=0.$ Therefore (\ref{eq:0<P1<P})
amounts to the following inequality
\begin{equation}
\label{eq:0<P_2<P}
0\leq aP(y)+a_4y_1^3y_2^2y_3+a_5y_1^3y_2y_3^2+a_6y_2^3y_1^2y_3+a_7y_2^3y_1y_3^2+a_8y_3^3y_1^2y_2+a_9y_3^3y_2^2y_1\leq P(y)
\end{equation}
Again, the equalities $P(1,1,1)=P(-1,1,1)$ and (\ref{eq:0<P_2<P}) imply $a_6+a_8=0.$ Thus if we sum inequalities (\ref{eq:0<P_2<P}) and the
resulting inequality in (\ref{eq:0<P_2<P}) when changing the sign of $y_1$ we get
$$|a_6y_1^2y_2y_3(y_2^2-y_3^2)|\leq \beta P(y)\quad\text{for some}\quad \beta\geq 0.$$
Taking $y_3=y_1$ in the last inequality we obtain
$$|a_6y_1^3y_2(y_2-y_1)(y_2+y_1)|\leq 2\beta y_1^2y_2^2(y_1-y_2)^2,$$
which implies $a_6=0$ if we let $y_1,y_2\to 1$ and $y_1\neq y_2.$ As the inequality is symmetric in the variables $a_i,$ then it is
straightforward to get $a_i=0$. Finally we get $P_1(y)=aP(y)$ which is a contradiction.

\end{proof}

\section{Extremal quadratic forms with orthotropic symmetry}
\label{sec:Extremal.orthot.symmetry}

The next theorem is the main result of the paper.

\begin{Theorem}
\label{th:sufficient.condition}
Assume the quadratic form $f(\xi)=(\xi+\xi^T)C(\xi+\xi^T)^T$ depending on the strain has orthotropic symmetry, i.e., the stiffness matrix $C$ has the form (\ref{stiffness.matrix}). Assume furthermore that $C_{11}C_{22}C_{33}\neq 0.$ If the determinant of the $y-$matrix of $f(x,y)$ is an extremal polynomial
that is not a perfect square, then $f$ is an extremal form.
\end{Theorem}

\begin{proof}
Assume in contradiction that $f(x,y)$ is not an extremal, then there exists a rank-one form $(x^TBy)^2$ such that
$$f(x,y)-(x^TBy)^2\geq 0\qquad\text{for all}\qquad x,y\in\mathbb R^3.$$
Let us now prove that then $f(x,y)=\alpha(x^TBy)^2$ for some $\alpha\geq 1.$
Recall the Brunn-Minkowski inequality for determinants [\ref{bib:Bec.Bel},\ref{bib:Bel}], which will be utilized in the sequel. 

\begin{Theorem}[Brunn-Minkowski inequality]
\label{th:Minkowsky}
Assume $n\in\mathbb N$ and $A$ and $B$ are $n\times n$ symmetric positive semi-definite matrices. Then the following inequality holds:
$$(\mathrm{det}(A+B))^{1/n}\geq(\mathrm{det}(A))^{1/n}+(\mathrm{det}(B))^{1/n}.$$
\end{Theorem}

Assume now $f(\xi)$ is a quasiconvex quadratic form that has a linear elastic orthotropic symmetry. Then $f$ has the form
$f(\xi)=(\xi+\xi^T)^TC(\xi+\xi^T),$ where the stiffness has the form of (\ref{stiffness.matrix}). Thus we get
$$f(\xi)=\sum_{i,j=1}^3{C_{ij}}\xi_{ii}\xi_{jj}+C_{44}(\xi_{12}+\xi_{21})^2+C_{55}(\xi_{13}+\xi_{31})^2+C_{66}(\xi_{23}+\xi_{32})^2.$$
It is clear that $f$ is then rank-one equivalent to a form
$$F(\xi)=\sum_{i,j=1}^3{a_{ij}}\xi_{ii}\xi_{jj}+a_1(\xi_{12}^2+\xi_{21}^2)+a_2(\xi_{13}^2+\xi_{31}^2)+a_3(\xi_{23}^2+\xi_{32}^2),$$
where $a_{ii}=C_{ii}$ and $a_i=C_{jj},$ with $j=i+3$ for $i=1,2,3.$
From the inequality
$$F(x,y)-(x^TBy)^2\geq 0$$
 we get that
 \begin{equation}
 \label{t-positive}
 F(x,y)-t(x^TBy)^2\geq 0\quad\text{ for all }\quad t\in [0,1].
 \end{equation}
Denote now by $T(y)$ the $y-$matrix of the biquadratic form $F(x,y)$ and by $T_t(y)$ the $y-$matrix of the biquadratic form $F(x,y)-t(x^TBy)^2.$
Inequality (\ref{t-positive}) now implies that the $y-$matrix of the form $F(x,y)-t(x^TBy)^2,$ i.e, the matrix $T_t(y)$ is positive semi-definite for all $y\in\mathbb R^3$ and $t\in[0,1].$ The equality
$$T(y)=[T(y)-T_t(y)]+[T_t(y)],$$
the positive semi-definiteness of the matrices $T(y)-T_t(y)$ and $T_t(y)$ and the Brunn-Minkowski inequality imply
$$\left(\mathrm{det}(T(y))\right)^{1/3}\geq \left(\mathrm{det}(T(y)-T_t(y))\right)^{1/3}+\left(\mathrm{det}(T_t(y))\right)^{1/3},$$
or
\begin{equation}
 \label{det.geq.det(t)}
\mathrm{det}(T(y))\geq\mathrm{det}(T_t(y)).
\end{equation}
It is easy to calculate that
\begin{equation}
\label{detreln}
\mathrm{det}(T_t(y))=\mathrm{det}(T(y))-t\sum_{i,j=1}^3l_{i}l_{j}\mathrm{cof}_{ij}(T(y)),
\end{equation}
where $l_i=\sum_{j=1}^3b_{ij}y_j.$
Inequality (\ref{det.geq.det(t)}) now implies
$$\mathrm{det}(T(y))\geq\mathrm{det}(T(y))-t\sum_{i,j=1}^3l_{i}l_{j}\mathrm{cof}_{ij}(T(y)).$$
As by the requirement of the theorem $\mathrm{det}(T(y))$ is not identically zero and for $t=0$ the right hand side of the last inequality is
exactly $\mathrm{det}(T(y))$, then by the extremality of $\mathrm{det}(T(y))$ the right hand side must be a multiple of $\mathrm{det}(T(y)),$ i.e.,
$$\mathrm{det}(T(y))-t\sum_{i,j=1}^3l_{i}l_{j}\mathrm{cof}_{ij}(T(y))=\lambda(t)\mathrm{det}(T(y)),$$
which gives
\begin{equation}
\label{sum.eq.det.t}
\sum_{i,j=1}^3l_{i}l_{j}\mathrm{cof}_{ij}(T(y))=\frac{1-\lambda(t)}{t}\mathrm{det}(T(y)),\qquad\text{for}\qquad t\neq 0.
\end{equation}
Both parts of the equality (\ref{sum.eq.det.t}) are polynomials in $y=(y_1,y_2,y_3)$ thus the expression $\frac{1-\lambda(t)}{t}$ must be constant, therefore
\begin{equation}
\label{sum.eq.det}
\sum_{i,j=1}^3l_{i}l_{j}\mathrm{cof}_{ij}(T(y))=d\cdot \mathrm{det}(T(y)),
\end{equation}
where $d\in\mathbb R.$ The positive semi-definiteness of $T(y)$ implies positive semi-definiteness of the cofactor matrix $\mathrm{cof}(T(y)),$ thus
$$\sum_{i,j=1}^3l_{i}l_{j}\mathrm{cof}_{ij}(T(y))\geq 0\qquad\text{for all}\qquad y\in\mathbb R^3.$$
We have on the other hand $\mathrm{det}(T(y))\geq0,$ and $\mathrm{det}(T(y))$ is not identically zero, thus $d\geq 0.$ Consider now two main cases:\\

\textbf{Case 1: $d=0.$} In this case identity (\ref{sum.eq.det}) becomes
$$
\sum_{i,j=1}^3l_{i}l_{j}\mathrm{cof}_{ij}(T(y))\equiv 0.
$$
Again, taking into account the positive semi-definiteness of $\mathrm{cof}(T(y))$ we get a system of three identities:
\begin{equation}
\label{syst.lin.cof}
l_1\mathrm{cof}_{i1}(T(y))+l_2\mathrm{cof}_{i2}(T(y))+l_3\mathrm{cof}_{i3}(T(y))\equiv 0,\qquad i=1,2,3.
\end{equation}
As the matrix $B$ is different from the zero matrix, it has a rank at least $1,$ thus the solution to the system of linear equations $l_i=0, i=1,2,3$ is a
proper subspace $V$ of $\mathbb R^3,$ i.e. is included in a hyperplane,
which means that the columns of the cofactor matrix $\mathrm{cof}(T(y))$ are linearly dependent in $\mathbb R^3\setminus V,$ i.e.,
 $\mathrm{det}(\mathrm{cof}(T(y)))=0$ for $y\in\mathbb R^3\setminus V.$ Therefore, since $\mathrm{det}(\mathrm{cof}(T(y)))$ is continuous in $\mathbb R^3$
it must be zero for all $y\in\mathbb R^3$ and by taking the determinant of the identity $T(y)[\mathrm{cof}(T(y))]^T=\mathrm{det}(T(y))I$ we get  $\mathrm{det}(T(y))\equiv 0,$  which is a contradiction.
\textbf{Case 1} is now proved.\\

\textbf{Case 2: $d>0.$}  In this case identity (\ref{sum.eq.det}) implies
\begin{equation}
\label{det.T.cof.k}
 \mathrm{det}(T(y))=k\sum_{i,j=1}^3l_{i}l_{j}\mathrm{cof}_{ij}(T(y)),\qquad k>0.
\end{equation}
Our goal is now getting a contradiction from $(\ref{det.T.cof.k})$.
Observe that
$$T(y)=
\begin{bmatrix}
a_{11}y_1^2+a_1y_2^2+a_2y_3^2 & a_{12}y_1y_2 & a_{13}y_1y_3\\
a_{12}y_1y_2 & a_{22}y_2^2+a_3y_3^2+a_1y_1^2 & a_{23}y_2y_3\\
a_{13}y_1y_3 & a_{23}y_2y_3 & a_{33}y_3^2+a_2y_1^2+a_3y_2^2
\end{bmatrix},
$$
thus we have the following formulae for the co-factors,
\begin{equation}
\label{cof.ii}
\mathrm{cof}_{ii}(T(y))=P_{ii}(y),
\end{equation}
and for $\ell\ne m$
\begin{equation}
\label{cof.ij}
\mathrm{cof}_{\ell m}(T(y))=y_\ell y_m P_{\ell m}(y),
\end{equation}
where $P_{ij}$ is a second or forth degree polynomial in $y$ depending only on $y_k^2$.
It follows that $\mathrm{det}(T(y))$ is a sixth order polynomial in $y$ depending only on $y_i^2$. Recalling the form of the entries of the cofactor matrix
and taking into account the fact that the coefficients of the expressions like $y_1^{\alpha_1}y_2^{\alpha_2}y_3^{\alpha_3}$ in the right hand side of (\ref{det.T.cof.k}) are zero, where at least one of the exponents $\alpha_i$ is odd, we get from (\ref{det.T.cof.k}) the following identity:

\begin{align}
\label{k.det}
\begin{split}
 \frac{1}{k}\mathrm{det}(T(y))&=(b_{11}^2y_1^2+b_{12}^2y_2^2+b_{13}^2y_3^2)\mathrm{cof}_{11}(T(y))\\
 &+(b_{21}^2y_1^2+b_{22}^2y_2^2+b_{23}^2y_3^2)\mathrm{cof}_{22}(T(y))\\
 &+(b_{31}^2y_1^2+b_{32}^2y_2^2+b_{33}^2y_3^2)\mathrm{cof}_{33}(T(y))\\
 &+2(b_{11}b_{22}+b_{12}b_{21})y_1y_2\mathrm{cof}_{12}(T(y))\\
 &+2(b_{11}b_{33}+b_{13}b_{31})y_1y_3\mathrm{cof}_{13}(T(y))\\
 &+2(b_{22}b_{33}+b_{23}b_{32})y_2y_3\mathrm{cof}_{23}(T(y)).\\
\end{split}
\end{align}
The right hand side of identity (\ref{k.det}) can be rearranged as follows:
\begin{align}
\label{k.det.rearranged}
\begin{split}
 \frac{1}{k}\mathrm{det}(T(y))&=(b_{11}y_1,b_{22}y_2,b_{33}y_3)\mathrm{cof}(T(y))(b_{11}y_1,b_{22}y_2,b_{33}y_3)^T+\\
 &+\left(b_{12}^2y_2^2\mathrm{cof}_{11}(T(y))+2b_{12}b_{21}y_1y_2\mathrm{cof}_{12}(T(y))+b_{21}^2y_1^2\mathrm{cof}_{22}(T(y))\right)\\
 &+\left(b_{13}^2y_3^2\mathrm{cof}_{11}(T(y))+2b_{13}b_{31}y_1y_3\mathrm{cof}_{13}(T(y))+b_{31}^2y_1^2\mathrm{cof}_{33}(T(y))\right)\\
 &+\left(b_{23}^2y_3^2\mathrm{cof}_{22}(T(y))+2b_{23}b_{32}y_2y_3\mathrm{cof}_{23}(T(y))+b_{32}^2y_2^2\mathrm{cof}_{33}(T(y))\right).\\
\end{split}
\end{align}
Since the cofactor matrix $\mathrm{cof}(T(y))$ is positive semi-definite, then each of the four summands is non-negative. Thus, the extremality
of the determinant $\mathrm{det}(T(y))$, that is a sum of four non-negative polynomials, implies that each summand is either identically zero or
a non-zero multiple of the determinant. On the other hand all four summands in (\ref{k.det.rearranged}) cannot be simultaneously identically zero.
Consider the following two cases:\\

\textbf{Case a: The first summand in (\ref{k.det.rearranged}) is a nonzero multiple of  $\mathrm{det}(T(y))$.} In this case we have the following representation of the determinant:
\begin{align}
\label{k.det.repr}
\begin{split}
 \frac{1}{s}\mathrm{det}(T(y))&= (b_{11}y_1,b_{22}y_2,b_{33}y_3)\mathrm{cof}(T(y))(b_{11}y_1,b_{22}y_2,b_{33}y_3)^T \\
& =
 \sum_{i,j=1}^3b_{ii}b_{jj}y_iy_j\mathrm{cof}_{ij}(T(y)),
\end{split}
\end{align}
where $s>0.$ Equating the coefficients of $y_1^6,$ $y_2^6$ and $y_3^6$ of the right and left hand sides of (\ref{k.det.repr}) we get
$a_{ii}=sb_{ii}^2,$ for $i=1,2,3.$
Now look at $G(\xi)=F(\xi)-s(\sum_{i=1}^3b_{ii}\xi_{ii})^2$. It has the form
$$G(\xi)=a_1(\xi_{12}^2+\xi_{21}^2)+a_2(\xi_{13}^2+\xi_{31}^2)+a_3(\xi_{23}^2+\xi_{32}^2)+2b_1\xi_{11}\xi_{22}+2b_2\xi_{11}\xi_{33}+2b_3\xi_{22}\xi_{33}.$$
and the determinant of its $y$-matrix is
\begin{align}
\label{det.T.G(y)}
\begin{split}
\mathrm{det}(T_{G}(y))&=(a_1^2a_2-b_1^2a_2)y_1^4y_2^2+(a_1^2a_3-b_1^2a_3)y_1^2y_2^4\\
&+(a_2^2a_1-b_2^2a_1)y_1^4y_3^2+(a_2^2a_3-b_2^2a_3)y_1^2y_3^4\\
&+(a_3^2a_1-b_3^2a_1)y_2^4y_3^2+(a_3^2a_2-b_3^2a_2)y_2^2y_3^4\\
&+2(a_1a_2a_3+b_1b_2b_3)y_1^2y_2^2y_3^2.
\end{split}
\end{align}
On the other hand by analogy with the formula (\ref{detreln})
(with $t=s$ and $B$ being diagonal) we have
$$
\mathrm{det}(T_{G}(y))=\mathrm{det}(T(y))-s\sum_{i,j=1}^3b_{ii}b_{jj}y_iy_j\mathrm{cof}_{ij}(T(y))=0.
$$

Consider now the three different subcases:\\

\textbf{Case a1: For all $i=1,2,3$ there holds $a_i>0.$}
It is clear that $\mathrm{det}(T_{G}(y))$ is identically zero if and only if $|b_i|=a_i$ and $a_1a_2a_3=-b_1b_2b_3.$ Therefore there are two cases possible: Either all $b_i$ are negative, or two of them are positive and the other one is negative. Note, that we can change the sign of any $x_i$ that will change the signs of two of $b_i,$ which means that one can without loss of generality assume that all $b_i$ are negative. Next, the substitution $x_i=\lambda x_i'$ and $y_i=\lambda y_i'$, where $\lambda_i\neq 0$ changes each $a_i$ by the factor $\lambda_i^2,$ transferring $G$ to a rank-one equivalent quadratic form of the same form with $a_i=1.$ The above non-singular transformations do not change the form of a linear combination of the variables $\xi_{11},$ $\xi_{22}$ and $\xi_{33},$ thus we end up with the formula for $F(\xi)$ up to rank-one equivalence:
\begin{equation}
\label{final.form.of.F1}
 F(\xi)=(a\xi_{11}+b\xi_{22}+c\xi_{33})^2+(\xi_{12}-\xi_{21})^2+(\xi_{13}-\xi_{31})^2+(\xi_{23}-\xi_{32})^2.
\end{equation}

It is straightforward to calculate
$$\mathrm{det}(T_{F}(y))=y_1^2(ay_1^2+by_2^2+cy_3^2)^2+y_2^2(ay_1^2+by_2^2+cy_3^2)^2+y_3^2(ay_1^2+by_2^2+cy_3^2)^2,$$
which is not an extremal polynomial unless $ay_1^2+by_2^2+cy_3^2$ is identically zero, i.e., $a=b=c=0,$ thus we get $C_{11}=0$
which is a contradiction. \textbf{Case a1} is proved.\\

\textbf{Case a2: $a_3=0,$ $a_1,a_2>0.$} In this case we again obtain from $\mathrm{det}(T_{G}(y))\equiv 0,$ that $|b_i|=a_i,$ for $i=1,2,3.$ The same argument as in the previous case leads to a situation
\begin{equation}
\label{final.form.of.F2}
 F(\xi)=(a\xi_{11}+b\xi_{22}+c\xi_{33})^2+(\xi_{12}+\xi_{21})^2+(\xi_{13}+\xi_{31})^2.
\end{equation}
Observe, that from the form of $F(\xi)$ we have
$$\mathrm{det}(T_{F}(y))=y_1^2P(y),$$
where $P(y)$ is a fourth degree polynomial of $y.$ Hilbert's theorem [\ref{bib:Hil}] asserts that any fourth degree non-negative
homogeneous polynomial in three variables is a sum of squares of degree two polynomials. Next, we have that $P(y)\geq 0,$ thus as $\mathrm{deg}(P)=4,$ then by
Hilbert's theorem $P(y)$ is a sum of squares of second degree polynomials, which means that
$\mathrm{det}(T_{F}(y))=y_1^2P(y)$ is either a perfect square or not an extremal, which is a contradiction. \textbf{Case a2} is proved.\\

\textbf{Case a3: $a_2=a_3=0,$ $a_1>0.$ } In this case $\mathrm{det}(T_{G}(y))\equiv 0$ implies $b_2=b_3=0$ thus we arrive at
 \begin{equation}
\label{final.form.of.F3}
 F(\xi)=(a\xi_{11}+b\xi_{22}+c\xi_{33})^2+\xi_{12}^2+\xi_{21}^2+2d\xi_{12}\xi_{21}.
\end{equation}
Like in the previous case, we again have $\mathrm{det}(T_{F}(y))=y_1^2P(y),$ and the same argument as in \textbf{Case a2}, completes the proof.\\

\textbf{Case a3: $a_1=a_2=a_3=0.$} In this case we have
$$F(\xi)=\sum_{i,j=1}^3a_{ij}\xi_{ii}\xi_{jj},$$
thus
$$\mathrm{det}(T_{G}(y))=ay_1^2y_2^2y_3^2,$$
which is again a contradiction. \textbf{Case a} is now completely proved.\\

\textbf{Case b: The second summand in (\ref{k.det.rearranged}) is a nonzero multiple of  $\mathrm{det}(T(y))$.}
 Observe, that we get in this case
 \begin{align}
\label{k.det.repr1}
 \frac{1}{s}\mathrm{det}(T(y))&=(b_{21}y_1,b_{12}y_2,0\cdot y_3)\mathrm{cof}(T(y))(b_{21}y_1,b_{12}y_2,0\cdot y_3)^T
\end{align}
for some $s>0,$ therefore \textbf{Case b} reduces to \textbf{Case a}. The theorem is proved now.
 \end{proof}

\section*{Acknowledgements}
D. Harutyunyan thanks Yury Grabovsky for useful discussion and Robert Lipton for his interest
and for pointing out Hilbert's theorem in [\ref{bib:Hil}]. The authors are grateful to the National Science
Foundation for support through grant DMS-1211359.

\end{document}